\documentclass[12pt,twoside,reqno]{amsart}
\usepackage[latin1]{inputenc}
\usepackage[italian,english]{babel}
\usepackage{amssymb,latexsym}
\usepackage{amsmath,amsfonts,amsthm}
\usepackage{paralist}
\usepackage{color}

\numberwithin{equation}{section}

\newcommand{\R}{\mathbb{R}}
\newcommand{\N}{\mathbb{N}}

\newcommand{\Z}{\mathbb{Z}}
\newcommand{\B}{\mathcal{B}}
\newcommand{\J}{\mathcal{J}}
\newcommand{\I}{\mathcal{I}}

\DeclareMathOperator{\h}{H^{s}(\R^{N})}

\newtheorem{lem}{Lemma}
\newtheorem{prop}{Proposition}
\newtheorem{thm}{Theorem}

\newtheorem{defn}{Definition} 
\theoremstyle{remark}
\newtheorem{remark}{Remark}

\begin{document}
\title[Ground states for superlinear fractional Schr\"odinger equations]{Ground states for superlinear fractional Schr\"odinger equations in $\R^{N}$}

\author{Vincenzo Ambrosio}
\address{Dipartimento di Matematica e Applicazioni "R. Caccioppoli"\\
         Universit\`a degli Studi di Napoli Federico II\\
         via Cinthia, 80126 Napoli, Italy}
\email{vincenzo.ambrosio2@unina.it}

\maketitle

\begin{abstract}
In this paper we study ground states of the following fractional Schr\"odinger equation
\begin{equation*}
\left\{
\begin{array}{ll}
(- \Delta)^{s} u + V(x) u = f(x, u) \, \mbox{ in } \, \R^{N},\\
u\in \h
\end{array}
\right.
\end{equation*}
where $s\in (0,1)$, $N>2s$ and $f$ is a continuous function satisfying a suitable growth assumption weaker than the Ambrosetti-Rabinowitz condition.\\
We consider the cases when the potential $V(x)$ is $1$-periodic or  has a bounded potential well.
\end{abstract}

\section{\bf Introduction}
\noindent
Recently there has been an increasing interest in the study of nonlinear partial differential equations driven by fractional operators, from a pure mathematical point of view as well as from concrete applications, since these operators naturally arise in several fields of research like obstacle problem, phase transition, conservation laws, financial market, flame propagations, ultra relativistic limits of quantum mechanic, minimal surfaces and water wave. The literature is too wide to attempt a reasonable list of references here, so we derive the reader to the work by Di Nezza, Patalluci and Valdinoci \cite{DPV}, where a more extensive bibliography and an introduction to the subject are given.\\
The present paper is devoted to the study of the following equation: 
\begin{equation}\label{P}
\left\{
\begin{array}{ll}
(- \Delta)^{s} u + V(x) u = f(x, u) \, \mbox{ in } \, \R^{N},\\
u\in \h
\end{array}
\right.
\end{equation}
where $s\in (0,1)$, $N>2s$, the potential $V(x)$  and the nonlinearity $f:\R^{N}\times \R\rightarrow \R$ satisfy the following assumptions:
\begin{compactenum}[(V1)]
\item $V \in C(\R^{N})$ and $\alpha \leq V(x) \leq \beta$; 
\end{compactenum}
\begin{compactenum}[(f1)]
\item $f\in C(\R^{N}\times \R)$ is $1$-periodic in $x$ and 
$$
\lim_{|t|\rightarrow \infty} \frac{f(x,t)}{|t|^{2^{*}_{s}-1}} =0 \quad \mbox{ uniformly in } x\in \R^{N}
$$  
where $2^{*}_{s}= \frac{2N}{N-2s}$; 
\item $f(x,t) = o(t)$ as $|t|\rightarrow 0$ uniformly in $x\in \R^{N}$.
\end{compactenum}
Here $\displaystyle{(- \Delta)^{s}}$ can be defined, for smooth functions $u$, by 
$$
(-\Delta)^{s} u(x) = c_{N,s}\, P.V. \int_{\R^{N}} \frac{u(x)-u(y)}{|x-y|^{N+2s}} \, dy,
$$
where P.V. stands for the Cauchy principal value and $c_{N,s}$ is a normalization constant; see \cite{CS1, DPV}.\\ 
Equation (\ref{P}) arises in the study of the Fractional Schr\"odinger equation
\begin{equation*}
\imath \frac{\partial \psi}{\partial t}+(- \Delta)^{s} \psi=H(x,\psi) \mbox{ in } \R^{N}\times \R
\end{equation*}
when the wave function $\psi$ is a standing wave, that is $\displaystyle{\psi(x,t)= u(x) e^{-\imath c t}}$, where $c$ is a constant.
This equation was introduced by Laskin \cite{Laskin1, Laskin2} and comes from an extension of the Feynman path integral from the Brownian-like to the Levy-like quantum mechanical paths. \\   
In recent years great attention has been focused on the fractional Schr\"odinger equation. 
Felmer, Quaas \& Tan \cite{FQT} studied the existence and regularity of positive solution to (\ref{P}) with $V(x)=1$ for general $s\in (0,1)$ when $f$ has subcritical growth and satisfies the Ambrosetti-Rabinowitz condition. Secchi \cite{Secchi1, Secchi2} proved some existence results for (\ref{P}) under the assumptions that the nonlinearity is either of perturbative type or satisfies the Ambrosetti-Rabinowitz condition. 
Cheng \cite{Cheng} proved the existence of bound state solutions for (\ref{P}) in which the potential $V(x)$ is unbounded and $f(x,u)=|u|^{p-1}u$ with $1<p<\frac{4s}{N}+1$.

\noindent
When $s=1$, formally, equation in (\ref{P}) reduces to the classical  Nonlinear Schr\"odinger Equation
\begin{equation}\label{NSE}
-\Delta u + V(x) u = f(x, u) \, \mbox{ in } \, \R^{N}, 
\end{equation}
which has been extensively studied in the last twenty years and we do not even try to review the huge bibliography.\\
To deal with (\ref{NSE}) many authors supposed that the nonlinear term satisfied the following condition due to Ambrosetti and Rabinowitz \cite{AR}
\begin{equation}
\exists \, \mu >2, \, R>0 : 0<\mu F(x,t) \leq f(x,t)t \quad \forall \, |t|\geq R,  \tag{AR}
\end{equation}
where $F$ is the primitive of $f$ with respect to the second variable. \\
This condition is very useful in critical point theory since it ensures the boundedness of the Palais-Smale sequences of the functional associated to (\ref{NSE}).
However, there are many functions which are superlinear at infinity, but do not satisfy (AR).
At this purpose, we would note that from the condition (AR) and the fact that $\mu>2$, it follows that
\begin{compactenum}[(f3)]
\item $\displaystyle{\lim_{|t|\rightarrow \infty}\frac{F(x,t)}{|t|^{2}}=+\infty}$, where $\displaystyle{F(x,t)=\int_{0}^{t} f(x, \tau) \, d\tau}$.
\end{compactenum}
Of course, also condition $(f3)$ characterizes the nonlinearity $f$ to be superlinear at infinity. 
It is easily seen that the function $f(x,t)=t \log(1+|t|)$ verifies $(f3)$ and does not satisfy (AR).
In order to study the nonlinear problem (\ref{NSE}) and to drop the condition (AR), Jeanjean in \cite{J} introduced the following assumption on $f$:
\begin{compactenum}[(f4)]
\item There exists $\lambda \geq 1$ such that
$$
G(x, \theta t)\leq \lambda G(x, t) 
\mbox{ for } (x, t)\in \R^{N}\times \R \mbox{ and } \theta\in [0,1],
$$
where $\displaystyle{G(x,t) = f(x,t)t - 2F(x,t)}$. 
\end{compactenum}

\noindent
The aim of this paper is to investigate solutions of the corresponding fractional case of problem (\ref{NSE}) without assuming (AR).
Since $u=0$ is a trivial solution to (\ref{P}) by $(f2)$, we will look for nontrivial solutions to (\ref{P}).

Our first result can be stated as follows
\begin{thm}\label{thm1}
Assume that $f$ satisfies $(f1)-(f4)$ and $V$ satisfies $(V1)$ and 
\begin{compactenum}[(V2)]
\item $V(x)$ is $1$-periodic.
\end{compactenum}  
Then there exists a nontrivial ground state solution $u\in \h$ to (\ref{P}). 
\end{thm}

\noindent
One of the main difficulty in studying (\ref{P}) is the nonlocal character of the fractional Laplacian $(-\Delta)^{s}$ with $s\in (0,1)$. To overcome this difficulty, Caffarelli and Silvestre \cite{CafSil} showed that it is possible to realize $(-\Delta)^{s}$  as an operator that maps a Dirichlet boundary condition to a Neumann boundary condition via an extension degenerate elliptic problem in $\R^{N+1}_{+}$. However, although this approach is very common nowadays (see \cite{CS1, CSS, FL, SV}), in this paper we prefer to investigate (\ref{P}) directly in $\h$ in order to apply the techniques used to study the case $s=1$.\\
More precisely, we will look for the critical points for the following functional
\begin{equation*}
\J(u)= \frac{1}{2} \Bigl[\iint_{\R^{2N}} \frac{|u(x)-u(y)|^{2}}{|x-y|^{N+2s}} dxdy  +\int_{\R^{N}} V(x) u^{2}(x) dx \Bigr]- \int_{\R^{N}} F(x,u) \, dx. 
\end{equation*}
\noindent
By assumptions on $f$ follow easily that $\J$ has a Mountain Pass geometry. Namely setting
$$
\Gamma = \{\gamma \in C([0,1], \h) : \gamma(0)=0 \, \mbox{ and } \J(\gamma(1))<0\} 
$$
we have $\Gamma\neq \emptyset$ and 
\begin{equation*}
c= \inf_{\gamma \in \Gamma} \max_{t\in [0,1]} \J(\gamma(t)).
\end{equation*}
The value $c$ is called the Mountain Pass level for $\J$.
Ekeland's principle \cite{Ek} guarantees the existence of a Cerami sequence at the level $c$. 
Hence, by using similar arguments to those developed in \cite{JT, Liu} and the $\Z^{N}$-invariant of the problem (\ref{P}), we will prove that every Cerami sequence for $\J$ is bounded and that there exists a subsequence which converges to a critical point for $\J$.\\
Finally, we will also consider the potential well case. We will assume that $V(x)$ satisfies, in addition to $(V1)$, the following condition 
\begin{compactenum}[(V3)]
\item $\displaystyle{V(x)< V_{\infty} := \lim_{|y|\rightarrow \infty} V(y)<\infty , \quad \forall x\in \R^{N}}$
\end{compactenum}
and that $f(x,u)=b(x) f(u)$ where $b\in C(\R^{N})$ and
\begin{equation}\label{3.11}
0<b_{\infty}:=\lim_{|y|\rightarrow \infty} b(y) \leq b(x) \leq \bar{b}<\infty
\end{equation}
for any $x\in \R^{N}$ and $f$ satisfies $(f1)-(f4)$. \\
Therefore our problem becomes
\begin{equation}\label{PP}
\left\{
\begin{array}{ll}
(-\Delta)^{s} u + V(x)u = b(x)f(u) \mbox{ in } \R^{N} \\
u\in \h
\end{array}.
\right.
\end{equation}
To study (\ref{PP}), we will use the energy comparison method in \cite{JT}. 
More precisely, introducing the energy functional at infinity 
\begin{equation*}
\J_{\infty}(u)= \frac{1}{2} \Bigl[\iint_{\R^{2N}} \frac{|u(x)-u(y)|^{2}}{|x-y|^{N+2s}} dxdy  +\int_{\R^{N}} V_{\infty} u^{2}(x) dx \Bigr] - \int_{\R^{N}} b_{\infty} F(u) \, dx
\end{equation*}
we will show that, under the above assumptions on $f$ and $V$, $\J$ has a nontrivial critical point provided that
\begin{equation}\label{minfty}
c<m_{\infty}
\end{equation} 
where
$$
m_{\infty} =\inf \{\J_{\infty}(u) : u\neq 0\,  \mbox{ and }\,  \J_{\infty}'(u)=0\}. 
$$
To prove (\ref{minfty}) we will exploit that our problem at infinity is autonomous  
$$
(-\Delta)^{s} u= -V_{\infty}u+b_{\infty}f(u) \, \mbox{ in } \R^{N},
$$
so it admits a least energy solution satisfying the Pohozaev identity; see \cite{ChangWang}.  This information will be useful to deduce the existence of a path $\gamma \in \Gamma$ such that $\displaystyle{\max_{t\in [0,1]} \J(\gamma(t))<m_{\infty}}$. 
Combining these facts, we will be able to prove our main second result:  

\begin{thm}\label{thm2}
Let $N> 2s$. Assume that $V$ satisfies $(V1)$ and $(V3)$, and that $f$ verifies the assumptions $(f1)-(f4)$. Then (\ref{PP}) has a ground state. 
\end{thm}

\noindent
The paper is organized as follows: In Section $2$ we introduce a variational setting of our problem and collect some preliminary results; in Section $3$ we prove the existence of a nontrivial ground state to (\ref{P}) when the potential $V$ is assumed $1$-periodic; finally, under the assumption that $V$ has a bounded potential well, we verify that it is possible to find a ground state to (\ref{PP}).

\section{\bf Preliminaries and functional setting}

\noindent
In this preliminary Section, for the reader's convenience, we collect some basic results that will be used in the forthcoming Sections.

Let us denote by $|\cdot|_{L^{q}(\R^{N})}$-the $L^{q}$ norm of a function $u:\R^{N}\rightarrow \R$.
We define the homogeneous fractional Sobolev space $\mathcal{D}^{s}(\R^{N})$ as the completion of $C^{\infty}_{c}(\R^{N})$ with respect to the norm
$$
||u||^{2}_{\mathcal{D}^{s}(\R^{N})} := \iint_{\R^{2N}} \frac{|u(x)-u(y)|^{2}}{|x-y|^{N+2s}} dx dy=[u]_{\h}^{2}.
$$
We denote by $\h$ the standard fractional Sobolev space, defined as the set of $u\in \mathcal{D}^{s}(\R^{N})$ satisfying $u\in L^{2}(\R^{N})$ with the norm
\begin{align}\label{n}
||u||_{\h}&:= \Bigl(\iint_{\R^{2N}} \frac{|u(x)-u(y)|^{2}}{|x-y|^{N+2s}} dx dy+\int_{\R^{N}} u^{2} dx \Bigr)^{\frac{1}{2}} \nonumber \\
&=[u]_{\h}^{2}+|u|_{L^{2}(\R^{N})}^{2}.  
\end{align}
For any $u\in \h$, it holds the following Sobolev inequality
\begin{equation*}
|u|_{L^{2^{*}_{s}}(\R^{N})} \leq C ||u||^{2}_{\mathcal{D}^{s}(\R^{N})}.
\end{equation*}

\noindent
Now, we recall the following lemmas which will be useful in the sequel. 
\begin{lem}\cite{DPV}\label{compemb}
$\h$ is continuously embedded in $L^{q}(\R^{N})$ for any $q\in [2, 2^{*}_{s}]$ and compactly embedded in $L^{q}_{loc}(\R^{N})$ for any $q\in [2, 2^{*}_{s})$. 
\end{lem}

\begin{lem}\cite{FQT}\label{lion}
Let $N>2s$. Assume that $\{u_{k}\}$ is bounded in $\h$ and it satisfies 
$$
\lim_{k\rightarrow +\infty} \sup_{\xi \in \R^{N}} \int_{B_{R}(\xi)} |u_{k} (x)|^{2} dx =0,
$$
where $R>0$. Then $u_{k}\rightarrow 0$ in $L^{q}(\R^{N})$ for $2<q<2^{*}_{s}$. 
\end{lem}

\noindent 
At this point, we give the definition of weak solution for the equation 
\begin{equation}\label{ws}
(-\Delta)^{s}u+V(x)u=g \mbox{ in } \R^{N}.
\end{equation}
\begin{defn}
Given $g\in L^{2}(\R^{N})$, we say that $u\in \h$ is a weak solution to (\ref{ws}) if $u$ satisfies
$$
\iint_{\R^{2N}} \frac{(u(x)-u(y))}{|x-y|^{N+2s}}(v(x)-v(y))\, dx dy + \int_{\R^{N}} V(x)uv \, dx = \int_{\R^{N}} gv \, dx
$$
for all $v\in \h$. 
\end{defn}

\noindent
To study solutions to (\ref{P}), we consider the following functional on $\h$ defined by setting
\begin{equation*}
\J(u)= \frac{1}{2}  \Bigl([u]_{\h}^{2}+\int_{\R^{N}} V(x)u^{2}(x) dx \Bigr) - \int_{\R^{N}} F(x,u) \, dx.
\end{equation*}
By $(V1)$ follows that
$$
[u]_{\h}^{2}+\int_{\R^{N}} V(x)|u(x)|^{2} dx
$$
is a norm which is equivalent to the standard norm defined in (\ref{n}). For such reason, we will always write 
$$
\J(u)= \frac{1}{2} ||u||^{2} - \int_{\R^{N}} F(x,u) \, dx.
$$
In particular, by assumptions on $f$, we deduce that $\J \in C^{1}(\h, \R)$.


 
\noindent
Let us observe that $\J$ possesses a Mountain Pass geometry. More precisely, we have the following result, whose simple proof is omitted.  

\begin{lem}\label{lem2.1}
Under the assumptions $(f1)-(f4)$,
there exist $r>0$ and $v_{0} \in \h$ such that $||v_{0}||>r$ and 
\begin{equation}\label{2.2}
b:= \inf_{||u||=r} \J(u) >\J(0)=0 \geq \J(v_{0}). 
\end{equation}
In particular 
\begin{align*}
&\langle{\J'(u), u}\rangle = ||u||^{2} + o(||u||^{2}) \mbox{ as } ||u||\rightarrow 0, \\
&\J(u)= \frac{1}{2} ||u||^{2} + o(||u||^{2})   \mbox{ as } ||u||\rightarrow 0
\end{align*}
and, as a consequence
\begin{compactenum}[(i)]
\item there exists $\eta>0$ such that if $v$ is a critical point for $\J$, then $||v||\geq \eta$; 
\item for any $c>0$ there exists $\eta_{c}>0$ such that if $\J(v_{n})\rightarrow c$ then $||v_{n}||\geq \eta_{c}$.   
\end{compactenum}
\end{lem}
\noindent

Therefore, by Lemma \ref{lem2.1}, follows that
$$
\Gamma = \{\gamma \in C([0,1], \h) : \gamma(0)=0 \, \mbox{ and } \J(\gamma(1))<0\} \neq \emptyset
$$
and we can define the Mountain Pass level 
\begin{equation}\label{2.3}
c= \inf_{\gamma \in \Gamma} \max_{t\in [0,1]} \J(\gamma(t)).
\end{equation}
Let us point out that, by (\ref{2.2}), $c$ is positive. Then, by using the Ekeland's principle \cite{Ek}, we know that 
there exists a Cerami sequence $\{v_{n}\}$ at the level $c$ for $\J$, that is
$$
\J(v_{n})\rightarrow c \, \mbox{ and } \, (1+||v_{n}||) ||\J'(v_{n})||_{*} \rightarrow 0. 
$$

\noindent
We conclude this section proving that the primitive $F(x,t)$ of $f(x,t)$ is nonnegative. 
\begin{lem}\label{F}
Let us assume that $f$ satisfies $(f1), (f2)$ and $(f4)$. Then $F\geq 0$ in $\R^{N} \times \R$.
\end{lem}

\begin{proof}
Firstly we observe that by $(f4)$ follows
$$
G(x, t) = f(x,t)t - 2F(x,t)\geq 0 \mbox{ for all } (x, t)\in \R^{N}\times \R. 
$$ 
Fix $t>0$. For $x\in \R^{N}$ let us compute the derivative of $\displaystyle{\frac{F(x,t)}{t^{2}}}$ with respect to $t$:
\begin{equation}\label{2.4}
\frac{\partial}{\partial t} \Bigl(\frac{F(x,t)}{t^{2}} \Bigr)= \frac{f(x,t) \,t^{2} - 2t\,F(x,t) }{t^{4}} \geq 0.
\end{equation} 
Moreover by $(f2)$ we get
\begin{equation}\label{2.5}
\lim_{t\rightarrow 0^{+}} \frac{F(x,t)}{t^{2}}=0. 
\end{equation}
Putting together (\ref{2.4}) and (\ref{2.5}) we deduce that $F(x,t)\geq0$ for all $(x, t)\in \R^{N}\times [0, +\infty)$. Analogously, we obtain that $F(x,t)\geq0$ for all $(x, t)\in \R^{N}\times (-\infty, 0]$. 

\end{proof}

\section{\bf Existence of ground states to (\ref{P})}

\noindent
In this Section we give the proof of the Theorem \ref{thm1}. 
We start proving the following Lemma, inspired by \cite{JT, Liu}, which guarantees the boundedness of Cerami sequences for the functional $\J$.  

\begin{lem}\label{lem2.3}
Assume that (V1), (f1), (f2), (f3) and (f4) hold true. Let $c\in \R$. Then any Cerami sequence for $\J$ is bounded. 
\end{lem}

\begin{proof}
Let $\{v_{n}\}$ be a Cerami sequence for $\J$.  \\
Assume by contradiction that $\{v_{n}\}$ is unbounded. Then going to a subsequence we may assume that
\begin{equation}\label{assurdo}
\J(v_{n})\rightarrow c, \quad ||v_{n}||\rightarrow \infty, \quad ||\J'(v_{n})||_{*}||v_{n}||\rightarrow 0. 
\end{equation}

\noindent
Now we define set $\displaystyle{w_{n}=\frac{v_{n}}{||v_{n}||}}$. Clearly $w_{n}$ is bounded in $\h$ and has unitary norm. 
We claim to prove that $\{w_{n}\}$ vanishes, i.e. it holds 
\begin{equation}\label{2.8}
\lim_{n\rightarrow \infty} \sup_{z\in \R^{N}} \int_{\B_{2}(z)} |w_{n}|^{2} dx=0. 
\end{equation} 
If (\ref{2.8}) does not hold, there exists $\delta>0$ such that
$$
\sup_{z\in \R^{N}} \int_{\B_{2}(z)} |w_{n}|^{2} dx \geq \delta >0. 
$$
As a consequence, we can choose $\{z_{n}\}\subset \R^{N}$ such that
$$
\int_{\B_{2}(z_{n})} |w_{n}|^{2} dx\geq \frac{\delta}{2}. 
$$
Since the number of points in $\Z^{N}\cap \B_{2}(z_{n})$ is less than $4^{N}$, then there exists $\xi_{n}\in \Z^{N}\cap \B_{2}(z_{n})$ such that 
\begin{equation}\label{2.9}
\int_{\B_{2}(\xi_{n})} |w_{n}|^{2} dx \geq K>0,
\end{equation}
where $K:=\delta 2^{-(2N+1)}$. Now we set $\tilde{w}_{n}= w_{n}(\cdot +\xi_{n})$. By using $(V1)$ and that $w_{n}$ has unitary norm, we deduce 
\begin{align*}
||\tilde{w}_{n}||^{2}&= [\tilde{w}_{n}]^{2}_{\h}+\int_{\R^{N}}V(x)|\tilde{w}_{n}|^{2} dx \\
&\leq  [\tilde{w}_{n}]^{2}_{\h} + \beta \int_{\R^{N}}V(x) |\tilde{w}_{n}(x)|^{2} dx\\
&= [w_{n}]^{2}_{\h}+ \beta \int_{\R^{N}} |w_{n}(x)|^{2} dx\\
&\leq \frac{\beta}{\alpha} \Bigl([w_{n}]^{2}_{\h}+ \alpha \int_{\R^{N}} |w_{n}(x)|^{2} dx  \Bigr) \\
&\leq \frac{\beta}{\alpha} \Bigl( [w_{n}]^{2}_{\h}+\int_{\R^{N}}V(x)|w_{n}|^{2} dx \Bigr) \\
&= \frac{\beta}{\alpha},
\end{align*}
that is $\tilde{w}_{n}$ is bounded. By Lemma \ref{compemb}, we may assume, going if necessary to a
subsequence, that 
\begin{equation}\begin{split}\label{limits}
&\tilde{w}_{n}\rightarrow \tilde{w} \mbox{ in } L^{2}_{loc}(\R^{N}), \\
&\tilde{w}_{n}(x) \rightarrow \tilde{w}(x) \mbox{ a.e. } x\in \R^{N}.  
\end{split}
\end{equation}
Then, by (\ref{2.9}) and (\ref{limits}) we get 
\begin{equation}\label{2.10}
\int_{\B_{2}(0)} |\tilde{w}|^{2} dx = \lim_{n\rightarrow \infty} \int_{\B_{2}(0)} |\tilde{w}_{n}|^{2} dx = \lim_{n\rightarrow \infty} \int_{\B_{2}(\xi_{n})} |w_{n}|^{2} dx \geq K>0,
\end{equation}
which implies $\tilde{w}\neq 0$. \\
Let $\displaystyle{\tilde{v}_{n} = ||v_{n}|| \tilde{w}_{n}}$. Since $\tilde{w}\neq 0$ 
 the set $A:= \{x\in \R^{N} : \tilde{w} \neq 0\}$ has positive Lebesgue measure and
$|\tilde{v}_{n}(x)|\rightarrow +\infty$.
In particular, by $(f3)$ we get
\begin{equation}\label{2.11}
\frac{F(x, \tilde{v}_{n}(x))}{|\tilde{v}_{n}(x)|^{2}} |\tilde{w}_{n}(x)|^{2} \rightarrow +\infty. 
\end{equation}
Let us observe that $f(x,t)$ is $1$-periodic with respect to $x$, so
\begin{equation}\label{2.12}
\int_{\R^{N}} F(x, v_{n}) \, dx = \int_{\R^{N}} F(x, \tilde{v}_{n}) \, dx. 
\end{equation}
By (\ref{assurdo}), (\ref{2.11}), (\ref{2.12}) and Lemma \ref{F} follow easily that
\begin{align}\label{2.13}
\frac{1}{2}- \frac{c+ o(1)}{||v_{n}||_{e}^{2}} &=\int_{\R^{N}} \frac{F(x, v_{n})}{||v_{n}||_{e}^{2}} dx \nonumber \\
&=  \int_{\R^{N}} \frac{F(x, \tilde{v}_{n})}{||v_{n}||_{e}^{2}} dx \nonumber \\
&\geq \int_{A} \frac{F(x, \tilde{v}_{n})}{|\tilde{v}_{n}|^{2}}|\tilde{w}_{n}|^{2} dx \rightarrow \infty
\end{align}
which gives a contradiction. Therefore (\ref{2.8}) holds true. In particular, by Lemma \ref{lion},  we get 
\begin{equation*}\label{2.14}
w_{n} \rightarrow 0 \, \mbox{ in } \,L^{q}(\R^{N}) \quad \forall q\in (2, 2^{*}_{s}). 
\end{equation*}
Now, let $\rho>0$ be a real number. By $(f1)-(f3)$ and Lemma \ref{F} follow that for any $\varepsilon >0$ there exists $C_{\varepsilon}>0$ such that
\begin{equation}\label{2.15}
0\leq F(x, \rho t) \leq \varepsilon (|t|^{2} + |t|^{2^{*}_{s}}) + C_{\varepsilon} |t|^{q}. 
\end{equation}
Since $||w_{n}||=1$, by Sobolev inequality we have that there exists $\tilde{c}>0$ such that 
\begin{equation}\label{2.15*}
|w_{n}|_{L^{2}(\R^{N})}^{2} +|w_{n}|_{L^{2^{*}_{s}}(\R^{N})}^{2^{*}_{s}}\leq \tilde{c}.
\end{equation}
Taking into account (\ref{2.15}) and (\ref{2.15*}) we have 
\begin{align*}
\limsup_{n\rightarrow \infty} \int_{\R^{N}} F(x, \rho w_{n}) \, dx &\leq \limsup_{n\rightarrow \infty}  [ \varepsilon (|w_{n}|_{L^{2}(\R^{N})}^{2} +|w_{n}|_{L^{2^{*}_{s}}(\R^{N})}^{2^{*}_{s}}) + C_{\varepsilon} (|w_{n}|_{L^{q}(\R^{N})}^{q}] \\
&\leq \varepsilon \tilde{c}
\end{align*}
and by the arbitrariness of $\varepsilon$  we get
\begin{equation}\label{2.16}
\lim_{n\rightarrow \infty} \int_{\R^{N}} F(x, \rho w_{n}) \, dx =0.
\end{equation}
Now, let $\{t_{n}\}\subset [0,1]$ be a sequence such that
\begin{equation}\label{2.17}
\J(t_{n} v_{n}) := \max_{t\in [0,1]} \J(tv_{n}). 
\end{equation}
By using (\ref{assurdo}) we can see that $\displaystyle{2\sqrt{j} ||v_{n}||^{-1} \in (0,1)}$ for $n$ sufficiently large and $j\in \N$. 
Taking $\rho= 2\sqrt{j}$ in (\ref{2.16}), we obtain
\begin{align*}
\J(t_{n}v_{n}) &\geq \J(2\sqrt{j} \, w_{n}) \\
&= 2j - \int_{\R^{N}} F(x, 2\sqrt{j} \, w_{n}) \, dx \geq j
\end{align*}
for $n$ large enough and for all $j\in \N$. 
Then 
\begin{equation}\label{Jinfty}
\J(t_{n}v_{n})\rightarrow +\infty.
\end{equation} 
Since $\J(0)=0$ and $\J(v_{n})\rightarrow c$ we deduce that $t_{n} \in (0,1)$.   By (\ref{2.17}) we get 
\begin{align}\label{J'0}
\langle{\J'(t_{n}v_{n}), t_{n}v_{n}}\rangle = t_{n} \frac{d}{dt} \J(tv_{n})\Bigr|_{t=t_{n}}=0.
\end{align} 
Indeed, putting together (\ref{assurdo}), (\ref{J'0}) and $(f4)$, we can see
\begin{align*}
\frac{2}{\lambda} \J(t_{n}v_{n}) &= \frac{1}{\lambda} \Bigl(2\J(t_{n}v_{n})  - \langle{\J'(t_{n}v_{n}), t_{n}v_{n}}\rangle\Bigr) \\
&=\frac{1}{\lambda} \int_{\R^{N}} \Bigl( f(x, t_{n}v_{n}) t_{n}v_{n} - 2F(x, t_{n}v_{n})\Bigr)\, dx \\
&= \frac{1}{\lambda} \int_{\R^{N}} G(x, t_{n}v_{n}) dx \\
&\leq \int_{\R^{N}} G(x, t_{n}v_{n}) dx \\
&= \int_{\R^{N}} \Bigl(f(x, v_{n}) v_{n} - 2 F(x, v_{n})\Bigr)\, dx \\
&= 2\J(v_{n})  - \langle{\J'(v_{n}), v_{n}}\rangle \rightarrow 2c
\end{align*}
which is incompatible with (\ref{Jinfty}). Thus $\{v_{n}\}$ is bounded. 

\end{proof}

\begin{remark}
Let us observe that the conclusion of Lemma \ref{lem2.3} holds true if we consider $\displaystyle{f(x,t)=b(x)f(t)}$ with $b\in C(\R^{N})$ and $0<b_{0}\leq b(x)\leq b_{1}<\infty$ for any $x\in \R^{N}$. In fact, in this case, the contradiction in (\ref{2.13}) follows by replacing (\ref{2.12}) by 
$$
\int_{\R^{N}} b(x) F(v_{n}) \, dx \geq \frac{b_{0}}{b_{1}} \int_{\R^{N}} b(x) F(\tilde{v}_{n}) \, dx. 
$$ 
\end{remark}
\medskip

\noindent
Now we prove that, up to a subsequence, our bounded Cerami sequence $\{u_{n}\}$ converges weakly to a non-trivial critical point for $\J$.

\begin{proof}[Proof of Theorem \ref{thm1}]
Let $c$ be the Mountain Pass level defined in (\ref{2.3}). We know that $c>0$ and that there exists a Cerami sequence $\{u_{n}\}$ for $\J$, which is bounded in $\h$ by Lemma \ref{lem2.3}. \\
We define
\begin{equation*}\label{2.18}
\delta:= \lim_{n\rightarrow \infty} \sup_{z\in \R^{N}} \int_{\B_{2}(z)} |u_{n}|^{2} dx. 
\end{equation*}
If $\delta=0$, then by Lemma \ref{lion} we have that $u_{n}\rightarrow 0$ in $L^{q}(\R^{N})$ for all $q\in (2, 2^{*}_{s})$. Analogously to (\ref{2.16}) we can see
\begin{equation*}\begin{split}\label{2.19}
& \lim_{n\rightarrow \infty} \int_{\R^{N}} F(x, u_{n}) \, dx =0, \\
& \lim_{n\rightarrow \infty} \int_{\R^{N}} f(x, u_{n}) u_{n} \, dx =0.
\end{split}
\end{equation*}
Then we deduce 
\begin{equation*}\label{2.7}
0=\lim_{n\rightarrow \infty} \int_{\R^{N}} \Bigl(\frac{1}{2} f(x, u_{n})v_{n} - F(x, u_{n})\Bigr)\, dx = \lim_{n\rightarrow \infty} \Bigl( \J(u_{n}) -\frac{1}{2} \langle{\J'(u_{n}), u_{n}}\rangle \Bigr)=c
\end{equation*}
which is impossible because of $c>0$. \\
Thus $\delta>0$. As for (\ref{2.10}), we can find a sequence $\{\xi_{n}\}\subset \Z^{N}$ and a positive constant $K$ such that
\begin{equation}\label{2.20}
\int_{\B_{2}(0)} |w_{n}|^{2} dx = \int_{\B_{2}(\xi_{n})} |u_{n}|^{2} dx >K
\end{equation}
where $w_{n}= u_{n}(\cdot + \xi_{n})$. Let us observe that $||w_{n}||= ||u_{n}||$, so $\{w_{n}\}$ is bounded. By Lemma \ref{compemb}, we can assume, up to a subsequence, that
\begin{align*}
&w_{n} \rightharpoonup w \, \mbox{ in } \, \h, \\ 
&w_{n} \rightarrow w \, \mbox{ in } \, L^{2}_{loc}(\R^{N})
\end{align*}
and by using (\ref{2.20}) we have $w\neq 0$. Since (\ref{P}) is $\Z^{N}$ invariant, $\{w_{n}\}$ is a Cerami sequence for $\J$. \\
Then, 
\begin{equation*}
\langle{\J'(w), \phi}\rangle = \lim_{n\rightarrow \infty} \langle{\J'(w_{n}), \phi}\rangle=0
\end{equation*}
for all $\phi \in C^{\infty}_{c}(\R^{N})$, that is $\J'(w)=0$ and $w$ is a nontrivial solution to (\ref{P}). \\
Now we want to prove that (\ref{P}) has a ground state. \\
Let
\begin{equation*}\label{2.21}
m= \inf\{ \J(v) : v\neq 0 \, \mbox{ and } \, \J'(v)=0\}
\end{equation*}
and suppose that $v$ is an arbitrary critical point for $\J$. By $(f4)$ we have
\begin{equation*}\label{2.22}
G(x,t) \geq 0 \quad \forall (x, t)\in \R^{N}\times \R
\end{equation*}
which implies that
\begin{equation*}\label{2.23}
\J(v)= \J(v)- \frac{1}{2} \langle{\J'(v), v}\rangle = \frac{1}{2} \int_{\R^{N}} G(x,v) \, dx \geq 0. 
\end{equation*}
Therefore $m\geq 0$.
Now, let $\{u_{n}\}$ be a sequence of nontrivial critical points for $\J$ such that $\J(u_{n})\rightarrow m$. By Lemma \ref{lem2.1} we have that for some $\eta >0$
\begin{equation}\label{2.24}
||u_{n}||\geq \eta. 
\end{equation}
Taking into account that $u_{n}$ is a critical point for $\J$ we have 
$$
(1+ ||u_{n}||)||\J'(u_{n})||_{*} \rightarrow 0.
$$  
Therefore $\{u_{n}\}$ is a Cerami sequence at the level $m$ and, by Lemma \ref{lem2.3}, $\{u_{n}\}$ is bounded in $\h$. \\
Let 
\begin{equation*}
\delta:= \lim_{n\rightarrow \infty} \sup_{z\in \R^{N}} \int_{\B_{2}(z)} |u_{n}|^{2} dx. 
\end{equation*}   
As before, if $\delta=0$ then
$$
\lim_{n\rightarrow \infty} \int_{\R^{N}} f(x, u_{n})u_{n} \, dx =0,
$$ 
from which 
\begin{equation}\label{2.25}
||u_{n}||^{2} = \langle{\J'(u_{n}), u_{n}}\rangle + \int_{\R^{N}} f(x, u_{n}) u_{n}\, dx \rightarrow 0,
\end{equation}
and this is impossible because of (\ref{2.24}). 
Thus $\delta>0$. The same argument made before proves that if we denote by $w_{n}(x)= u_{n}(x+\xi_{n})$ we deduce that 
\begin{equation}\label{Jm}
\J'(w_{n})=0, \, \, \J(w_{n})= \J(u_{n})\rightarrow m
\end{equation}
and $w_{n}$ weakly converges to a nonzero critical point $w$ for $\J$. \\
Thus, by (\ref{Jm}), $G\geq 0$ and Fatou Lemma follow that
\begin{equation}\begin{split}\label{2.26}
\J(w) &= \J(w) - \frac{1}{2} \langle{\J'(w), w}\rangle \\
&= \frac{1}{2}\int_{\R^{N}} G(x, w) \, dx\\
&\leq \liminf_{n\rightarrow \infty} \frac{1}{2} \int_{\R^{N}} G(x, w_{n}) \, dx\\
&= \liminf_{n\rightarrow \infty} \Bigl(\J(w_{n}) -\frac{1}{2} \langle{\J'(w_{n}), w_{n}}\rangle \Bigr) =m.
\end{split}
\end{equation}
Hence $w$ is a nontrivial critical point for $\J$ such that $\J(w)=m$. This concludes the proof of the Theorem.

\end{proof}

\section{\bf Proof of Theorem \ref{thm2}}
\noindent
In the last section we give the proof of the Theorem \ref{thm2}. We proceed as in \cite{JT, Liu}.
The main ingredient of our proof is the following result which takes advantage of the Pohozaev identity proved in \cite{ChangWang}:
\begin{prop}\label{prop3.1}
Let $u\in \h$ be a nontrivial critical point for 
\begin{equation*}
\I(u)=\frac{1}{2} [u]_{\h}^{2} -\int_{\R^{N}} G(u) \, dx. 
\end{equation*}
Then there exists $\gamma \in C([0,1], \h)$ such that $\gamma(0)=0$, $\I(\gamma(1))<0$, $u\in \gamma([0,1])$ and
$$
\max_{t\in [0,1]} \I(\gamma(t)) = \I(u). 
$$
\end{prop}

\begin{proof}
Let $u\in \h$ be a nontrivial critical point for $\I$.
We set for $t>0$
\begin{equation*}\label{3.3}
u^{t}(x) = u\Bigl(\frac{x}{t}\Bigr). 
\end{equation*} 
By using Pohozaev identity in \cite{ChangWang}, we know  
$$
\frac{N-2s}{2} [u]^{2}_{\h}=N\int_{\R^{N}}G(u) dx,
$$
so we can see that 
\begin{align*}
\I(u^{t})&= \frac{t^{N-2s}}{2} [u]_{\h}^{2}- t^{N} \int_{\R^{N}} G(u) \,dx \\
&= \Bigl(\frac{1}{2} t^{N-2s} - \frac{N-2s}{2N}t^{N}\Bigr) [u]_{\h}^{2}.  
\end{align*}
Therefore we can deduce that $\displaystyle{\max_{t>0} \I(u^{t})= \I(u)}$, $\I(u^{t})\rightarrow -\infty$ as $t\rightarrow \infty$, and
$$
||u^{t}||_{\h}^{2}=t^{N-2s}[u]_{\h}^{2}+ t^{N} |u|_{L^{2}(\R^{N})}^{2}\rightarrow 0 \mbox{ as } t\rightarrow 0. 
$$ 
Choosing $\alpha>1$ such that $\I(u^{\alpha})<0$ and setting 
\begin{equation*}
\gamma(t)= 
\left\{
\begin{array}{ll}
u^{\alpha t} &\mbox{ for } t\in (0,1]\\
0 &\mbox{ for } t=0.   
\end{array}
\right.
\end{equation*}
we get the conclusion. 

\end{proof}

\noindent
Now we consider the following functionals 
\begin{equation*}
\J(u)= \frac{1}{2} ||u||^{2} - \int_{\R^{N}} b(x) F(u) \, dx
\end{equation*}
and
\begin{equation*}
\J_{\infty}(u)= \frac{1}{2} \Bigl([u]_{\h}^{2}+\int_{\R^{N}} V_{\infty} u^{2}(x) dx\Bigr) - \int_{\R^{N}} b_{\infty} F(u) \, dx. 
\end{equation*}
By $(V3)$ follows that 
\begin{equation}\label{3.5}
\J(u)<\J_{\infty}(u) \mbox{ for any } u\in \h\setminus \{0\}. 
\end{equation}
Taking into account of the Proposition \ref{prop3.1}, we can prove the following
\begin{lem}\label{JT}
Let $N> 2s$. Assume that $V(x)$ satisfies $(V1)$ and $(V3)$ and $f$ satisfies $(f1)-(f4)$.
Then $\J$ has a nontrivial critical point. 
\end{lem}
\begin{proof}
Let $c$ be the Mountain Pass level for $\J$. We know that $\J$ has a Cerami sequence $\{u_{n}\}$ at the level $c$, which is bounded by Lemma \ref{lem2.3}. Then, by Lemma \ref{compemb}, follows that $u_{n}\rightharpoonup u$ in $\h$ and $\J'(u)=0$. We claim to prove that $u\not\equiv 0$.\\
Assume by contradiction that $u= 0$. Taking into account $(V3)$, $u_{n}$ converges to $u$ in $L^{2}_{loc}(\R^{N})$ and (\ref{3.11}) we can deduce that 
\begin{align*}
|\J_{\infty}(u_{n}) - \J(u_{n})|\leq \int_{\R^{N}} [V_{\infty}-V(x)] u_{n}^{2} dx + \int_{\R^{N}} [b_{\infty}-b(x)] F(u_{n}) dx \rightarrow 0
\end{align*}
and
\begin{align*}
||\J_{\infty}(u_{n}) - \J(u_{n})||&\leq \sup_{\overset{\phi \in \h}{||\phi||_{\h}=1}} \Bigl\{\Bigl|\int_{\R^{N}} [V_{\infty}-V(x)] u_{n} \phi dx\Bigr| \\
&+ \Bigl|\int_{\R^{N}} [b_{\infty}-b(x)] f(u_{n}) \phi dx\Bigr| \Bigr\}\rightarrow 0
\end{align*}
that is $u_{n}$ is a Palais-Smale sequence for $\J_{\infty}$ at the level $c$. \\
Now we define 
\begin{equation}\label{3.6}
\delta:= \lim_{n\rightarrow \infty} \sup_{\xi \in \R^{N}} \int_{B_{2}(\xi)} u_{n}^{2} dx.
\end{equation}
If $\delta=0$, proceeding similarly to (\ref{2.25}), we deduce that $||u_{n}||^{2}\rightarrow 0$ which contradicts with Lemma \ref{lem2.1}. So, $\delta>0$ and there exists $\{\xi_{n}\}\subset \Z^{N}$ such that
\begin{equation}\label{3.7}
\int_{\B_{2}(\xi_{n})} |u_{n}|^{2} dx \geq \frac{\delta}{2}>0. 
\end{equation}
Let $v_{n}= u_{n}(x+\xi_{n})$. Then 
\begin{align*}
& ||v_{n}||= ||u_{n}||,\\
& \J_{\infty}(v_{n})= \J_{\infty}(u_{n}),\\ 
&\J'_{\infty}(v_{n})= \J'_{\infty}(u_{n}).
\end{align*} 
Therefore $\{v_{n}\}$ is a bounded Palais-Smale sequence for $\J_{\infty}$. As in the proof of Theorem \ref{thm1}, by (\ref{3.7}) we deduce that $v_{n}\rightharpoonup v$ in $\h$ and $v$ is a nontrivial critical point for $\J_{\infty}$. \\
Moreover, proceeding as in (\ref{2.26}) we have
\begin{align*}
\J_{\infty}(v) \leq c.
\end{align*}
Now, by using Proposition \ref{prop3.1} with $\displaystyle{g(t)= b_{\infty}f(t)-V_{\infty}t}$, we deduce the existence of $\gamma_{\infty} \in C( [0,1], \h)$ such that $\gamma_{\infty}(0)=0$, $\J_{\infty}(\gamma_{\infty}(1))<0$, $v\in \gamma_{\infty}([0,1])$ and 
\begin{equation*}
\max_{t\in [0,1]} \J_{\infty}(\gamma_{\infty}(t)) = \J_{\infty}(v). 
\end{equation*}
Since $0\notin \gamma_{\infty}((0,1])$, by (\ref{3.5}) follows that, for all $t\in (0,1]$
\begin{equation}\label{3.8}
\J(\gamma_{\infty}(t))<\J_{\infty}(\gamma_{\infty}(t)). 
\end{equation}
In particular $\J(\gamma_{\infty}(1))\leq \J_{\infty}(\gamma_{\infty}(1))<0$, so $\gamma_{\infty} \in \Gamma$. Then, taking into account  $\J(0)=\J_{\infty}(0)=0$, (\ref{3.8}) and $c>0$, we deduce that
\begin{equation*}\label{3.9}
c\leq \max_{t\in [0,1]} \J(\gamma_{\infty}(t)) < \max_{t\in [0,1]} \J_{\infty}(\gamma_{\infty}(t))= \J_{\infty}(v)\leq c
\end{equation*}
which gives a contradiction.

\end{proof}

\begin{remark}\label{remark2}
Let us observe that being $\{u_{n}\}$ a Cerami sequence for $\J$ at the level $c$ and $u_{n}\rightharpoonup u$ in $\h$, by using a similar argument as in (\ref{2.26}), we can deduce that $\J(u)\leq c$. 
\end{remark}

\noindent
Finally, we give the proof of Theorem \ref{thm2}
\begin{proof}[Proof of Theorem \ref{thm2}]

Let $\displaystyle{m=\inf \{\J(u) : u\neq 0 \mbox{ and } \J'(u)=0\}}$ and we denote by $u$ the nontrivial critical point for $\J$ obtained in the previous Lemma. \\
Then (see Remark \ref{remark2}) we can see 
\begin{equation}\label{3.10}
0\leq m \leq \J(u)\leq c. 
\end{equation}
Now, let $\{u_{n}\}$ be a sequence of nontrivial critical points for $\J$ such that $\J(u_{n})\rightarrow m$. As in the proof of Theorem \ref{thm1}, we have that $\{u_{n}\}$ is a Cerami bounded sequence at the level $m$ and $\delta>0$, where $\delta$ is defined via (\ref{3.6}). \\
Extracting a subsequence, $u_{n}\rightharpoonup \tilde{u}$ in $\h$, and $\tilde{u}$ is a critical point for $\J$ satisfying $\J(\tilde{u})\leq m$ as in (\ref{2.26}). \\
Now, if $\tilde{u}=0$, $\{u_{n}\}$ is a bounded Palais-Smale sequence for $\J_{\infty}$ at the level $m$. Since $\delta>0$, we deduce that $v_{n}$, which is a suitable translation of $\{u_{n}\}$, converges weakly to some critical point $v\neq 0$ for $\J_{\infty}$ and $\J_{\infty}(v)\leq m$. \\
Proceeding similarly to the proof of Lemma \ref{JT},  by Proposition \ref{prop3.1} follows that there exists $\gamma_{\infty} \in \Gamma_{\infty}\cap \Gamma$ such that
\begin{equation*}
c\leq \max_{t\in [0,1]} \J(\gamma_{\infty}(t)) < \max_{t\in [0,1]} \J_{\infty}(\gamma_{\infty}(t))= \J_{\infty}(v)\leq m
\end{equation*}     
which is a contradiction because of (\ref{3.10}). As a consequence $\tilde{u}$ is a nontrivial critical point for $\J$ such that $\J(\tilde{u})=m$. 
\end{proof}


\begin{thebibliography}{777}


\bibitem{AR}
A. Ambrosetti and P. H. Rabinowitz,
{\it Dual Variational Methods in Critical Point Theory and Applications},
J. Funct. Anal. {\bf 14} (1973), 349--381.

\bibitem{CS1}
X. Cabr{{\'e}} and Y.Sire,
{\it Nonlinear equations for fractional Laplacians I: regularity, maximum principles, and Hamiltonian estimates},
Ann. Inst. H. Poincare Anal. Non Lineaire {\bf 31} (2014), 23--53.

\bibitem{CSS}
L.Caffarelli, S.Salsa and L.Silvestre, 
{\it Regularity estimates for the solution and the free boundary of the obstacle problem for the fractional {L}aplacian},
Invent. Math. {\bf 171} (2008), 425--461.

\bibitem{CafSil}
L.A. Caffarelli and L.Silvestre,
{\it An extension problem related to the fractional Laplacian},
Comm. Partial Differential Equations {\bf 32} (2007),1245--1260.

\bibitem{ChangWang}
X. J. Chang and Z.-Q. Wang, 
{\it Ground state of scalar field equations involving fractional {L}aplacian with general nonlinearity}, Nonlinearity {\bf 26} (2013), 479--494.

\bibitem{Cheng}
M. Cheng, 
{\it Bound state for the fractional Schr\"odinger equation with unbounded potential},
J. Math. Phys. {\bf 53}, 043507 (2012).

\bibitem{DPV}
E. Di Nezza, G. Palatucci, E. Valdinoci,
{\it Hitchhiker's guide to the fractional Sobolev spaces},
Bull. Sci. Math. {\bf 136} (2012), 521--573.

\bibitem{Ek}
I. Ekeland, 
{\it Convexity methods in Hamiltonian Mechanics}, 
Springer (1990).

\bibitem{FQT}
P. Felmer, A. Quaas, and J. G. Tan, 
{\it Positive solutions of nonlinear Schr\"odinger equation with the fractional {L}aplacian}
Proc. Roy. Soc. Edinburgh Sect. A {\bf 142} (2012), no. 6, 1237--1262.

\bibitem{FL}
R. Frank and E. Lenzmann, 
{\it Uniqueness and nondegeneracy of ground states for $(-\Delta)^{s} Q+Q-Q^{\alpha +1}$ in $\R$},
e-print arXiv:1009.4042.

\bibitem{J}
L. Jeanjean, 
{\it On the existence of bounded Palais-Smale sequences and application to a Landesman- Lazer type problem set on $\R^{N}$},
Proc. Roy. Soc. Edinburgh Sect.A, {\bf129} (1999), 787--809.

\bibitem{JT}
L. Jeanjean and K. Tanaka, 
{\it A positive solution for an asymptotically linear elliptic problem on $\mathbb{R}^N$ autonomous at infinity},
ESAIM Control Optim. Calc. Var. {\bf 7}, 597--614 (2002).


\bibitem{Laskin1}
N. Laskin,
{\it Fractional quantum mechanics and L\`evy path integrals}, 
Phys. Lett. A {\bf 268} (2000), 298--305.

\bibitem{Laskin2}
N. Laskin, 
{\it Fractional Schr\"odinger equation}, 
Phys. Rev. E {\bf 66} (2002), 056108.

\bibitem{Liu}
S. B. Liu,  
{\it On ground states of superlinear $p$-Laplacian equations in $\mathbb{R}^{N}$},
J. Math. Anal. Appl. {\bf 361} (2010), 48--58.

\bibitem{Secchi1}
S. Secchi, 
{\it Ground state solutions for nonlinear fractional Schr\"odinger equations in $\R^{N}$}, 
J. Math. Phys. {\bf 54} (2013), 031501.

\bibitem{Secchi2}
S. Secchi, 
{\it Perturbation results for some nonlinear equations involving fractional operators},
Differ. Equ. Appl. {\bf 5} (2013), no. 2, 221--236

\bibitem{SV}
Y. Sire and E. Valdinoci,
{\it Fractional {L}aplacian phase transitions and boundary reactions: a geometric inequality and a symmetry result},
J. Funct. Anal. {\bf 256} (2009), 1842--1864.





\end{thebibliography}
\end{document}